\documentclass[a4paper, 11pt]{amsart}
\usepackage[utf8]{inputenc}
\usepackage[english]{babel}
\usepackage[T1]{fontenc}
\usepackage{amsmath, amssymb, amsfonts}
\usepackage[all]{xy}
\usepackage{enumerate}
\usepackage{graphicx}
\usepackage{comment}
\usepackage[margin=2.25cm]{geometry}
\usepackage{hyperref}

\title{Permanence of approximation properties for discrete quantum groups}
\author{Amaury Freslon}
\keywords{Quantum groups, approximation properties, relative amenability}
\subjclass[2010]{20G42, 46L65}
\address{Univ. Paris Diderot, Sorbonne Paris Cité, IMJ-PRG, UMR 7586 CNRS, Sorbonne Universit\'es, UPMC Univ. Paris 06, F-75013, Paris, France}
\email{amaury.freslon@imj-prg.fr}
\date{}

\theoremstyle{plain}
\newtheorem{thm}{Theorem}[section]
\newtheorem{prop}[thm]{Proposition}
\newtheorem{cor}[thm]{Corollary}
\newtheorem{lem}[thm]{Lemma}

\newtheorem*{thmetoile}{Theorem}

\theoremstyle{definition}
\newtheorem{de}[thm]{Definition}
\newtheorem{ex}[thm]{Example}

\theoremstyle{remark}
\newtheorem{rem}[thm]{Remark}

\DeclareMathOperator{\Ad}{Ad}
\DeclareMathOperator{\Id}{Id}
\DeclareMathOperator{\Ir}{Irr}

\DeclareMathOperator{\Pol}{Pol}
\DeclareMathOperator{\Supp}{Supp}

\DeclareMathOperator{\Tr}{Tr}

\renewcommand{\i}{\imath}
\renewcommand{\tilde}{\widetilde}

\newcommand{\B}{\mathcal{B}}
\newcommand{\C}{\mathbb{C}}
\newcommand{\D}{\Delta}
\newcommand{\E}{\mathbb{E}}

\newcommand{\G}{\mathbb{G}}
\newcommand{\HH}{\mathbb{H}}
\newcommand{\K}{\mathbb{K}}

\newcommand{\N}{\mathbb{N}}

\newcommand{\RR}{\mathcal{R}}

\newcommand{\Z}{\mathbb{Z}}

\newcommand{\h}{\widehat}
\newcommand{\w}{\omega}

\begin{document}

\begin{abstract}
We prove several results on the permanence of weak amenability and the Haagerup property for discrete quantum groups. In particular, we improve known facts on free products by allowing amalgamation over a finite quantum subgroup. We also define a notion of relative amenability for discrete quantum groups and link it with amenable equivalence of von Neumann algebras, giving additional permanence properties.
\end{abstract}

\maketitle

\section{Introduction}

A fruitful way of studying compact quantum groups is through their dual discrete quantum group. Geometric group theory is then a rich source of inspiration, even though results can seldom be straightforwardly transferred from the classical to the quantum setting. A very important property from this point of view is of course \emph{amenability}, the study of which culminated in R. Tomatsu's work \cite{tomatsu2006amenable}. However, some important examples of discrete quantum groups, often called "free quantum groups", are known not to be amenable by \cite[Cor 5]{banica1997groupe}. One should then look for weak versions of amenability like the \emph{Haagerup property} or \emph{weak amenability}. Some of these \emph{approximation properties} were investigated for unimodular discrete quantum groups by J. Kraus and Z.-J. Ruan in \cite{kraus1999approximation}, but no genuinely quantum example was found.

There recently was a regain of interest in the subject thanks to M. Brannan's breakthrough paper \cite{brannan2011approximation} proving the Haagerup property for the free orthogonal and free unitary quantum groups $O_{N}^{+}$ and $U_{N}^{+}$. On the one hand, his techniques were extended to free wreath products by F. Lemeux in \cite{lemeux2013fusion} and to more general classes of unitary easy quantum groups by the author in \cite{freslon2013fusion}. On the other hand, weak amenability was also established for $O_{N}^{+}$ and $U_{N}^{+}$ by the author in \cite{freslon2012examples} and later extended, with K. De Commer and M. Yamashita in \cite{freslon2013ccap} using monoidal equivalence.

At that point, it seemed reasonable to try to establish a general theory of approximation properties for discrete quantum groups. For the Haagerup property, this was done by M. Daws, P. Fima, A. Skalski and S. White in \cite{daws2014haagerup} using the more general setting of locally compact quantum groups. For weak amenability, this was part of the author's PhD thesis \cite{freslon2013proprietes}. In particular, we gave in the latter several permanence properties, i.e. group-theoretic constructions preserving weak amenability. One of the most important is the following one : if two discrete quantum groups are weakly amenable with Cowling-Haagerup constant equal to $1$, then their free product is again weakly amenable with Cowling-Haagerup constant equal to $1$, which was proved in \cite{freslon2012note}.

However, there are several other natural constructions to look at like direct products or inductive limits. Moreover, the techniques of E. Ricard and X. Qu \cite{ricard2005khintchine}, which are crucial in the proof of the aforementioned result on free products, also work for amalgamated free products under certain assumptions. For discrete groups, it is not very difficult to see that these assumptions are satisfied as soon as the amalgam is finite. That this also works for quantum groups will be proved in Section \ref{sec:permanence}. Note that even if it was certainly known to experts, the result for classical groups did not appear yet in the literature. Let us also mention that our proofs are based on multiplier techniques, so that they can also be applied to the Haagerup property, improving \cite[Prop 7.13]{daws2014haagerup} by removing the unmodularity assumption.

The last permanence property which we will study is \emph{relative amenability}. In fact, this was the first to be used in combination with weak amenability by M. Cowling and U. Haagerup in \cite[Thm 6.4]{cowling1989completely} to prove that lattices in different symplectic groups yield non-isomorphic von Neumann algebras. Their proof relied on \cite[Prop 6.2]{cowling1989completely} :

\begin{thmetoile}[Cowling-Haagerup]
Let $G$ be a locally compact group and let $\Gamma$ be a lattice in $G$. Then, $\Lambda_{cb}(\Gamma) = \Lambda_{cb}(G)$, where $\Lambda_{cb}$ denotes the \emph{Cowling-Haagerup constant}.
\end{thmetoile}

The core of the proof is the finite covolume assumption, which is however too restrictive in our context. In \cite{eymard1972moyennes}, P. Eymard defined  a subgroup $H\subset G$ to be \emph{co-amenable} if there exists a mean on the homogeneous space $G/H$ (i.e. a state on $\ell^{\infty}(G/H)$) which is invariant with respect to the translation action of $G$ (so lattices in locally compact groups are in particular co-amenable). He investigated this property as a weakening of the notion of amenability since a group is amenable if and only if all its subgroups are co-amenable. It was proved in \cite[Paragraphe 4.10]{anantharaman1995amenable} that if $H$ is co-amenable in $G$, then $\Lambda_{cb}(H) = \Lambda_{cb}(G)$. Since this proof makes use of the machinery of von Neumann algebras and correspondences, it may be used as a starting point for a generalization to quantum groups, which will be done in Section \ref{sec:relative}.

Let us briefly outline the content of this work. In Section \ref{sec:preliminaries}, we introduce the necessary material concerning quantum groups and their actions as well as approximation properties. In Section \ref{sec:permanence} we give several permanence results. The main one is the stability of weak amenability under free products amalgamated over finite quantum subgroups, if the Cowling-Haagerup constant is equal to $1$. Finally, Section \ref{sec:relative} deals with relative amenability. We define it and prove that it is equivalent to amenable equivalence of the associated von Neumann algebras. We then study the simplest examples, namely finite index quantum subgroups and eventually prove permanence properties for relatively amenable discrete quantum groups.

\subsection*{Acknowledgments}

The results of this work were part of the author's PhD thesis and he wishes to thank his advisor E. Blanchard for his supervision. The author is also thankful to S. Vaes for pointing out a mistake in a preliminary version, to C. Anantharaman-Delaroche for discussions on amenable equivalence and to A. Skalski and the anonymous referee for their useful comments.

\section{Preliminaries}\label{sec:preliminaries}

In this paper, we will denote by $\otimes$ the tensor product of Hilbert spaces and the minimal tensor product of C*-algebras, and by $\overline{\otimes}$ the tensor product of von Neumann algebras. Once the spaces involved are clear, we will denote tensor products of elements or maps by $\otimes$.

\subsection{Discrete quantum groups}

Discrete quantum groups will be seen as duals of compact quantum groups in the sense of S.L. Woronowicz. We briefly present this theory as introduced in \cite{woronowicz1995compact}, i.e. in the C*-algebraic setting. We will then explain how one passes to the von Neumann algebraic setting, which will prove more convenient when dealing with relative amenability (see Remark \ref{rem:relativeC*}). In the sequel, all tensor products of C*-algebras are minimal.

\begin{de}
A \emph{compact quantum group} $\G$ is a pair $(C(\G), \Delta)$ where $C(\G)$ is a unital C*-algebra and $\Delta : C(\G)\rightarrow C(\G)\otimes C(\G)$ is a unital $*$-homomorphism such that
\begin{equation*}
(\Delta\otimes \i)\circ\Delta = (\i\otimes\Delta)\circ\Delta
\end{equation*}
and both $\Delta(C(\G))(1\otimes C(\G))$ and $\Delta(C(\G))(C(\G)\otimes 1)$ span dense subspaces of $C(\G)\otimes C(\G)$.
\end{de}

The main feature of compact quantum groups is the existence of a \emph{Haar state}, which is both left and right invariant (see \cite[Thm 1.3]{woronowicz1995compact}).

\begin{prop}
Let $\G$ be a compact quantum group. Then, there exists a unique state $h$ on $\G$, called the Haar state, such that for all $a\in C(\G)$,
\begin{eqnarray*}
(\i\otimes h)\circ \D(a) =h(a).1 \\
(h\otimes \i)\circ \D(a) =h(a).1
\end{eqnarray*}
\end{prop}

Let $(L^{2}(\G), \xi_{h})$ be the associated GNS construction and let $C_{\text{red}}(\G)$ be the image of $C(\G)$ under the GNS map, called the \emph{reduced form} of $\G$. Let $W$ be the unique (by \cite[Thm 4.1]{woronowicz1995compact}) unitary operator on $L^{2}(\mathbb{G})\otimes L^{2}(\mathbb{G})$ satisfying, for all for $\xi \in L^{2}(\mathbb{G})$ and $a\in C(\mathbb{G})$,
\begin{equation*}
W^{*}(\xi\otimes a\xi_{h}) = \Delta(a)(\xi\otimes \xi_{h})
\end{equation*} 
and let $\h{W} := \Sigma W^{*}\Sigma$. Then, $W$ is a \emph{multiplicative unitary} in the sense of \cite{baaj1993unitaires}, i.e. $W_{12}W_{13}W_{23} = W_{23}W_{12}$ and we have the following equalities :
\begin{equation*}
C_{\text{red}}(\mathbb{G}) = \overline{\rm{span}}(\i\otimes \mathcal{B}(L^{2}(\G))_{*})(W) \text{ and } \Delta(x) = W^{*}(1\otimes x)W
\end{equation*}
for all $x\in C_{\text{red}}(\mathbb{G})$. Moreover, we can define the \emph{dual discrete quantum group} $\h{\G} = (C_{0}(\h{\G}), \h{\Delta})$ by
\begin{equation*}
C_{0}(\h{\mathbb{G}}) = \overline{\rm{span}}(\mathcal{B}(L^{2}(\G))_{*}\otimes\i)(W) \text{ and } \h{\Delta}(x) = \h{W}^{*}(1\otimes x)\h{W}
\end{equation*}
for all $x\in C_{0}(\mathbb{G})$. The two von Neumann algebras associated to these quantum groups are
\begin{equation*}
L^{\infty}(\G) = C_{\text{red}}(\G)'' \text{ and } \ell^{\infty}(\h{\G})=C_{0}(\h{\G})''
\end{equation*}
where the bicommutants are taken in $\B(L^{2}(\G))$. The coproducts extend to normal maps on these von Neumann algebras and one can prove that $W\in L^{\infty}(\G)\overline{\otimes} \ell^{\infty}(\h{\G})$. Moreover, the Haar state of $\G$ extends to a normal state on $L^{\infty}(\G)$. In order to give an alternative description of $\ell^{\infty}(\h{\G})$, we need to define representations of quantum groups.

\begin{de}
Let $\G$ be a compact quantum group.
\begin{itemize}
\item A \emph{representation} of $\G$ on a Hilbert space $H$ is an invertible operator $U\in L^{\infty}(\G)\overline{\otimes} \B(H)$ such that $(\D\otimes \i)(U) = U_{13}U_{23}$.
\item A \emph{representation} of $\h{\G}$ on a Hilbert space $H$ is an invertible operator $U\in \ell^{\infty}(\h{\G})\overline{\otimes} \B(H)$ such that $(\h{\D}\otimes \i)(U) = U_{13}U_{23}$.
\item A representation is said to be \emph{unitary} if the operator $U$ is unitary.
\end{itemize}
\end{de}

The linear span of coefficients of unitary representations of $\G$ forms a Hopf-$*$-algebra $\Pol(\G)$ which is dense in $C(\G)$. Let $\Ir(\G)$ be the set of isomorphism classes of irreducible unitary representations of $\G$ (which are all finite-dimensional by \cite[Thm 1.2]{woronowicz1995compact}). If $\alpha\in \Ir(\G)$, we will denote by $u^{\alpha}$ a representative of $\alpha$ and by $H_{\alpha}$ the finite-dimensional Hilbert space on which it acts. There are $*$-isomorphisms
\begin{equation*}
C_{0}(\h{\G}) = \bigoplus_{\alpha\in \Ir(\G)}\B(H_{\alpha})\text{ and }\ell^{\infty}(\h{\G}) = \prod_{\alpha\in \Ir(\G)}\B(H_{\alpha}).
\end{equation*}
The minimal central projection in $\ell^{\infty}(\h{\G})$ corresponding to the identity of $\B(H_{\alpha})$ will be denoted by $p_{\alpha}$ and there exist positive matrices $(Q_{\alpha})_{\alpha\in \Ir(\G)}$ such that the two normal semi-finite faithful (in short n.s.f.) weights
\begin{eqnarray}\label{eq:haarweight1}
h_{L} : x & \mapsto & \sum_{\alpha\in \Ir(\G)}\Tr(Q_{\alpha})\Tr(Q_{\alpha}(xp_{\alpha})) \\ \label{eq:haarweight2}
h_{R} : x & \mapsto & \sum_{\alpha\in \Ir(\G)}\Tr(Q_{\alpha}^{-1})\Tr(Q_{\alpha}^{-1}(xp_{\alpha})) \\ \nonumber
\end{eqnarray}
on $\ell^{\infty}(\h{\G})$ are respectively left and right invariant. These so-called \emph{Haar weights} are both unique up to multiplication by a scalar.

\subsection{Actions on von Neumann algebras}

Actions of quantum groups can be defined both on C*-algebras and on von Neumann algebras. However, only the latter will be used in this paper. The main feature, which will prove of importance later on, is the existence of a unitary implementation. The standard reference on this subject is \cite{vaes2001unitary}.

\begin{de}
A \emph{(left) action} of $\h{\G}$ on a von Neumann algebra $M$ is a unital normal $*$-homomor\-phism $\rho : M\rightarrow \ell^{\infty}(\h{\G})\overline{\otimes} M$ such that
\begin{equation*}
(\h{\D}\otimes \i)\circ\rho = (\i\otimes \rho)\circ\rho.
\end{equation*}
\end{de}

The \emph{fixed point algebra} of the action $\rho$ is the subalgebra $M^{\rho} = \{x\in M, \rho(x) = 1\otimes x\}$. A subalgebra $N$ of $M$ is said to be \emph{stable} under the action if $\rho(N)\subset \ell^{\infty}(\h{\G})\overline{\otimes} N$. In that case, there is a \emph{restricted action} of $\h{\G}$ on $N$ which will still be denoted by $\rho$. The crossed-product construction in this setting generalizes the classical definition.

\begin{de}
Let $\h{\G}$ be a discrete quantum group acting on a von Neumann algebra $M$. The \emph{crossed-product} $\h{\G}\ltimes_{\rho} M$ is the von Neumann subalgebra of $\B(L^{2}(\G))\overline{\otimes} M$ generated by $\rho(M)$ and $L^{\infty}(\G)\otimes 1$.
\end{de}

The crossed-product is endowed with a \emph{dual action} $\h{\rho}$ of $\G^{\text{op}}$ (i.e. with respect to the flipped coproduct) defined by
\begin{equation*}
\left\{\begin{array}{ccc}
\h{\rho}(\rho(m)) & = & 1\otimes \rho(m) \\
\h{\rho}(a\otimes 1) & = & [(\sigma\circ\D)(a)]\otimes 1
\end{array}\right.
\end{equation*}
for all $m\in M$ and $a\in L^{\infty}(\G)$. Let $\h{\G}$ be a discrete quantum group and let $M$ be a von Neumann algebra together with a fixed n.s.f. weight $\theta$ with GNS construction $(K, \imath, \Lambda_{\theta})$. It is proven in \cite{vaes2001unitary} that any action of $\h{\G}$ on $M$ is unitarily implementable, i.e. there exists a unitary
\begin{equation*}
U^{\rho}\in \ell^{\infty}(\G)\overline{\otimes} \B(K)
\end{equation*}
which is the adjoint of a representation of $\h{\G}$ and such that
\begin{equation*}
\rho(x) = U^{\rho}(1\otimes x)(U^{\rho})^{*}
\end{equation*}
for all $x\in M$.

\subsection{Quantum subgroups and quotients}

Let us give some details concerning the notions of discrete quantum subgroups and quotients, which will appear in all the constructions of this work. Let $\G$ be a compact quantum group and let $\HH$ be another compact quantum group such that $C(\HH)\subset C(\G)$ and the coproduct of $\HH$ is given by the restriction of the coproduct of $\G$. Then, $\h{\HH}$ is said to be a \emph{discrete quantum subgroup} of $\h{\G}$. The following important fact was proved in \cite[Prop 2.2]{vergnioux2004k}.

\begin{prop}[Vergnioux]\label{prop:conditionalexpectationsubgroup}
Let $\h{\G}$ be a discrete quantum group, let $\h{\HH}$ be a discrete quantum subgroup and denote the respective Haar states of $\G$ and $\HH$ by $h_{\G}$ and $h_{\HH}$. Then, there exists a faithful conditional expectation $\E_{\HH} : C_{\text{red}}(\G)\rightarrow C_{\text{red}}(\HH)$ such that $h_{\HH}\circ\E_{\HH} = h_{\G}$. Moreover, $\E_{\HH}$ extends to a faithful conditional expectation from $L^{\infty}(\G)$ to $L^{\infty}(\HH)$, still denoted by $\E_{\HH}$.
\end{prop}

Note that the inclusion $C(\HH)\subset C(\G)$ extends to inclusions of matrix algebras $M_{n}(C(\HH))\subset M_{n}(C(\G))$ so that any finite-dimensional representation of $\HH$ can be seen as a finite-dimensional representation of $\G$. This inclusion obviously preserves intertwiners, so that we have an inclusion $\Ir(\HH)\subset \Ir(\G)$. Let us define a central projection
\begin{equation*}
p_{\HH} = \sum_{\alpha\in \Ir(\HH)} p_{\alpha}\in \ell^{\infty}(\h{\G}).
\end{equation*}
We can use $p_{\HH}$ to describe the structure of $\h{\HH}$ from the structure of $\h{\G}$ (see \cite[Prop 2.3]{fima2008kazhdan} for details).

\begin{prop}\label{prop:subgroup}
With the notations above, we have
\begin{enumerate}
\item $\h{\D}(p_{\HH})(p_{\HH}\otimes 1) = p_{\HH}\otimes p_{\HH}$
\item $\ell^{\infty}(\h{\HH}) = p_{\HH}\ell^{\infty}(\h{\G})$
\item $\h{\D}_{\HH}(a) = \h{\D}(a)(p_{\HH}\otimes p_{\HH})$
\item If $h_{L}$ is a left Haar weight for $\h{\G}$, then $h_{L, \HH} : x\mapsto h_{L}(xp_{\HH})$ is a left Haar weight for $\h{\HH}$.
\end{enumerate}
\end{prop}

It is easy to see, using the above statements, that the map $a\mapsto \h{\D}(a)(1\otimes p_{\HH})$ defines a right action (right actions of discrete quantum groups are defined in the same way as left actions with the obvious modifications) of $\h{\HH}$ on $\ell^{\infty}(\h{\G})$. Let $\ell^{\infty}(\h{\G}/\h{\HH})$ be the fixed point subalgebra for this action. Using Proposition \ref{prop:subgroup} again, we see that the restriction of the coproduct $\h{\D}$ to $\ell^{\infty}(\h{\G}/\h{\HH})$ yields a left action of $\h{\G}$ on this von Neumann algebra which will be denoted by $\tau$. Let $h_{L}$ be the left-invariant weight of $\h{\G}$ defined by Equation (\ref{eq:haarweight1}). It is known from \cite[Prop 2.4]{fima2008kazhdan} that the map
\begin{equation*}
T : x \mapsto (\i\otimes h_{L, \HH})[\h{\D}(x)(1\otimes p_{\HH})]
\end{equation*}
is a normal faithful operator-valued weight from $\ell^{\infty}(\h{\G})$ to $\ell^{\infty}(\h{\G}/\h{\HH})$ and that there exists a n.s.f. weight $\theta$  on $\ell^{\infty}(\h{\G}/\h{\HH})$ such that $h_{L} = \theta\circ T$. Let $U$ be the unitary implementation of the action $\tau$ with respect to the weight $\theta$. Then, $\RR= U^{*}$ will be called the \emph{quasi-regular representation} of $\h{\G}$ modulo $\h{\HH}$.

\begin{rem}\label{rem:invariantweight}
By a straighforward calculation, we see that when both sides are well-defined,
\begin{equation*}
\tau\circ T(x) = (\i\otimes \i\otimes h_{L, \HH})[(\h{\D}\otimes \i)(\h{\D}(x)(1\otimes p_{\HH}))] = (\i\otimes T)\circ\tau(x).
\end{equation*}
The weight $\theta$ can therefore be interpreted as an almost invariant measure on the quotient space with respect to the action $\tau$.
\end{rem}

\subsection{Approximation properties}

Two approximation properties will be considered in this paper : weak amenability and the Haagerup property. They have both been defined in earlier works and enjoy various characterizations (see for example \cite{kraus1999approximation}, \cite{freslon2012note} and \cite{daws2014haagerup}). For our purpose, the point of view of \emph{multipliers} is the best suited. We refer the reader to \cite[Ch 12]{brown2008finite} for an introduction to approximation properties for classical groups, which motivates the following definitions.

\begin{de}\label{de:quantummultiplier}
Let $\h{\G}$ be a discrete quantum group and $a\in \ell^{\infty}(\h{\G})$. The \emph{left multiplier} associated to $a$ is the map $m_{a} : \Pol(\G) \rightarrow \Pol(\G)$ defined by
\begin{equation*}
(m_{a}\otimes \i)(u^{\alpha}) = (1\otimes ap_{\alpha})u^{\alpha},
\end{equation*}
for every irreducible representation $\alpha$ of $\G$. A net $(a_{t})$ of elements of $\ell^{\infty}(\h{\G})$ is said to \emph{converge pointwise} to $a\in \ell^{\infty}(\h{\G})$ if for every irreducible representation $\alpha$ of $\G$, $a_{t}p_{\alpha}\rightarrow ap_{\alpha}$ in $\B(H_{\alpha})$. An element $a\in \ell^{\infty}(\h{\G})$ is said to have \emph{finite support} if $ap_{\alpha} = 0$ for all but finitely many $\alpha\in \Ir(\G)$.
\end{de}

\begin{de}\label{de:quantumwa}
A discrete quantum group $\h{\G}$ is said to be \emph{weakly amenable} if there exists a net $(a_{t})$ of elements of $\ell^{\infty}(\h{\G})$ such that
\begin{itemize}
\item $a_{t}$ has finite support for all $t$.
\item $(a_{t})$ converges pointwise to $1$.
\item $K:=\limsup_{t} \|m_{a_{t}}\|_{cb}$ is finite.
\end{itemize}
The lower bound of the constants $K$ for all nets satisfying these properties is denoted by $\Lambda_{cb}(\h{\G})$ and called the \emph{Cowling-Haagerup constant} of $\h{\G}$. By convention, $\Lambda_{cb}(\h{\G})=\infty$ if $\h{\G}$ is not weakly amenable.
\end{de}

\begin{de}\label{de:quantumhaagerup}
A discrete quantum group $\h{\G}$ is said to have the \emph{Haagerup property} if there exists a net $(a_{t})$ of elements of $\ell^{\infty}(\h{\G})$ such that
\begin{itemize}
\item $a_{t}\in C_{0}(\h{\G})$ for all $t$.
\item $(a_{t})$ converges pointwise to $1$.
\item $m_{a_{t}}$ is completely positive for all $t$.
\end{itemize}
\end{de}

As in the classical case, these properties are connected to corresponding approximation properties for the associated operator algebras (see for instance \cite[Ch 12]{brown2008finite} for the definitions). However, the link is not fully understood yet when the quantum group is not \emph{unimodular}. Let us therefore only state results in the unimodular case (i.e. when the Haar state is a trace), which were proved in \cite[Thm 5.14]{kraus1999approximation}.

\begin{thm}[Kraus-Ruan]\label{thm:quantumwa}
Let $\h{\G}$ be a \emph{unimodular} discrete quantum group. Then,
\begin{itemize}
\item $\h{\G}$ has the Haagerup property $\Leftrightarrow$ $C_{\text{red}}(\G)$ has the Haagerup property relative to $h$ $\Leftrightarrow$ $L^{\infty}(\G)$ has the Haagerup property.
\item $\Lambda_{cb}(\h{\G}) = \Lambda_{cb}(C_{\text{red}}(\G)) = \Lambda_{cb}(L^{\infty}(\G))$.
\end{itemize}
\end{thm}

\section{Permanence results}\label{sec:permanence}

This section is divided into two parts. In the first one, we will prove permanence of approximation properties under several elementary constructions. In the second part, we will extend the result of \cite{freslon2012note} to free products amalgamated over a finite quantum subgroup. The result also holds for the Haagerup property, thus improving \cite[Thm 7.8]{daws2014haagerup} by removing the unimodularity assumption.

\subsection{First results}

We start with the simplest case, namely passing to quantum subgroups. The permanence of the Haagerup property by passing to discrete quantum subgroups (or more generally to closed quantum subgroups of a locally compact quantum group) was proved in \cite[Prop 5.7]{daws2014haagerup}. We therefore only consider weak amenability.

\begin{prop}\label{cor:wasubgroup}
Let $\h{\G}$ be a weakly amenable discrete quantum group and let $\h{\HH}$ be a discrete quantum subgroup of $\h{\G}$. Then, $\h{\HH}$ is weakly amenable and $\Lambda_{cb}(\h{\HH}) \leqslant \Lambda_{cb}(\h{\G})$.
\end{prop}

\begin{proof}
Let $\epsilon > 0$ and let $(a_{t})$ be a net of finitely supported elements in $\ell^{\infty}(\h{\G})$ converging pointwise to $1$ and such that $\limsup \|m_{a_{t}}\|_{cb}\leqslant \Lambda_{cb}(\h{\G}) + \epsilon$. Then, using the notations of Proposition \ref{prop:subgroup}, $(a_{t}p_{\HH})$ is a net of finitely supported elements in $\ell^{\infty}(\h{\HH})$ which converges pointwise to $1$. Using the conditional expectation of Proposition \ref{prop:conditionalexpectationsubgroup}, we see that $\|m_{a_{t}p_{\HH}}\|_{cb} = \|\E_{\HH}\circ m_{a_{t}}\|_{cb} \leqslant \|m_{a_{t}}\|_{cb}$. Thus, $\Lambda_{cb}(\h{\HH}) \leqslant \Lambda_{cb}(\h{\G}) + \epsilon$.
\end{proof}

It is proved in \cite{wang1995tensor} that if $\h{\G}$ and $\h{\HH}$ are two discrete quantum groups, then the minimal tensor product $C(\G)\otimes C(\HH)$ can be turned into a compact quantum group in a natural way. Its dual discrete quantum group is denoted by $\h{\G}\times \h{\HH}$ and called the \emph{direct product} of $\h{\G}$ and $\h{\HH}$.

\begin{prop}\label{cor:directproduct}
Let $\h{\G}$ and $\h{\HH}$ be two discrete quantum groups. Then, $\h{\G}\times \h{\HH}$ is weakly amenable if and only if both $\h{\G}$ and $\h{\HH}$ are weakly amenable. Moreover,
\begin{equation*}
\max(\Lambda_{cb}(\h{\G}), \Lambda_{cb}(\h{\HH}))\leqslant\Lambda_{cb}(\h{\G}\times \h{\HH}) \leqslant \Lambda_{cb}(\h{\G})\Lambda_{cb}(\h{\HH}).
\end{equation*}
\end{prop}

\begin{proof}
The "only if" part is a direct consequence of Proposition \ref{cor:wasubgroup}, as well as the first inequality. To prove the second inequality, let $\epsilon > 0$ and let $(a_{t})$ and $(b_{s})$ be nets of finitely supported elements respectively in $\ell^{\infty}(\h{\G})$ and in $\ell^{\infty}(\h{\HH})$ converging pointwise to $1$ and such that $\limsup \|m_{a_{t}}\|_{cb} \leqslant \Lambda_{cb}(\h{\G}) +\epsilon$ and $\limsup \|m_{b_{s}}\|_{cb} \leqslant \Lambda_{cb}(\h{\HH}) +\epsilon$. Set
\begin{equation*}
c_{(t, s)} = a_{t}\otimes b_{s} \in \ell^{\infty}(\h{\G}\times \h{\HH}).
\end{equation*}
From the description of the representation theory of direct products given in \cite[Thm 2.11]{wang1995tensor}, we see that $(c_{(t, s)})$ is a net of finitely supported elements converging pointwise to $1$. Moreover, since $m_{c_{(t, s)}} = m_{a_{t}}\otimes m_{b_{s}}$, we have $\Lambda_{cb}(\h{\G}\times \h{\HH})\leqslant (\Lambda_{cb}(\h{\G}) +\epsilon)(\Lambda_{cb}(\h{\HH}) +\epsilon)$, which concludes the proof.
\end{proof}

\begin{rem}
It is a general fact that for any two C*-algebras $A$ and $B$, $\Lambda_{cb}(A\otimes B) = \Lambda_{cb}(A)\Lambda_{cb}(B)$ (see e.g. \cite[Thm 12.3.13]{brown2008finite}). Hence, we always have
\begin{equation*}
\Lambda_{cb}(C_{\text{red}}(\G\otimes \HH)) = \Lambda_{cb}(C_{\text{red}}(\G))\Lambda_{cb}(C_{\text{red}}(\HH)).
\end{equation*}
Moreover, Theorem \ref{thm:quantumwa} implies that $\Lambda_{cb}(\h{\G}\times \h{\HH}) = \Lambda_{cb}(\h{\G})\Lambda_{cb}(\h{\HH})$ as soon as the discrete quantum groups are unimodular. It is very likely that this equality holds in general but we were not able to prove it.
\end{rem}

A similar statement holds for the Haagerup property.

\begin{prop}
Let $\h{\G}$ and $\h{\HH}$ be two discrete quantum groups. Then, $\h{\G}\times \h{\HH}$ has the Haagerup property if and only if both $\h{\G}$ and $\h{\HH}$ have the Haagerup property.
\end{prop}

\begin{proof}
The "only if" part comes from stability under passing to quantum subgroups.To prove the "if" part, one can build, as in the proof of Proposition \ref{cor:directproduct}, multipliers on $\h{\G}\times\h{\HH}$ by tensoring multipliers on the two quantum groups. To finish the proof, note that a tensor product of unital completely positive maps is again unital and completely positive and that the tensor product of two elements in $C_{0}(\h{\G})$ and $C_{0}(\h{\HH})$ respectively lies in $C_{0}(\h{\G}\times \h{\HH})$.
\end{proof}

The third construction we will study is inductive limits of discrete quantum groups, or equivalently inverse limits of compact quantum groups. It was proved in \cite[Prop 3.1]{wang1995free} that given a family of discrete quantum groups $(\h{\G}_{i})_{i}$ with connecting maps $\pi_{ij} : C(\G_{i}) \rightarrow C(\G_{j})$ intertwining the coproducts and satisfying $\pi_{jk}\circ\pi_{ij} = \pi_{ik}$, there is a natural compact quantum group structure on the inductive limit C*-algebra. Its dual discrete quantum group is called the \emph{inductive limit} of the system $(\h{\G}_{i}, \pi_{ij})$. In order to study approximation properties, we first need to understand its representation theory. Since we were not able to find a reference for this, we give a statement even though it is certainly well-known to experts.

\begin{prop}\label{prop:inductivelimit}
Let $(\h{\G}_{i}, \pi_{ij})$ be an inductive system of discrete quantum groups with inductive limit $\h{\G}$ and assume that all the maps $\pi_{ij}$ are injective. Then, there is a one-to-one correspondence between the irreducible representations of $\G$ and the increasing union of the sets of irreducible representations of the $\G_{i}$'s.
\end{prop}

\begin{proof}
The maps $\pi_{ij}$ being injective, we can identify each $\h{\G}_{i}$ with a discrete quantum subgroup of the $\h{\G}_{j}$'s for $j\geqslant i$. This gives inclusions of the sets of irreducible representations and we denote by $\mathcal{S}$ the increasing union of these sets. We can also identify each $C(\G_{i})$ with a C*-subalgebra of $C(\G)$ in such a way that
\begin{equation*}
\overline{\bigcup C(\G_{i})} = C(\G).
\end{equation*}
Under this identification, the discrete quantum groups $\h{\G}_{i}$ are quantum subgroups of $\h{\G}$, hence any irreducible representation of some $\G_{i}$ yields an irreducible representation of $\G$ and we have proved that $\mathcal{S}\subset \Ir(\G)$. Moreover, the algebra
\begin{equation*}
\mathcal{A} := \bigcup_{i}\Pol(\G_{i})
\end{equation*}
is a dense Hopf-$*$-subalgebra of $C(\G)$ spanned by coefficients of irreducible representations. Because of Schur's orthogonality relations, the density implies that the coefficients of all irreducible representations of $\G$ are in $\mathcal{A}$, i.e. $\mathcal{A} = \Pol(\G)$. This means that any irreducible representation of $\G$ comes from an element of $\mathcal{S}$ and $\Ir(\G) = \mathcal{S}$.
\end{proof}

We can now prove the permanence of weak amenability under this construction.

\begin{prop}\label{cor:wainductivelimits}
Let $(\h{\G}_{i}, \pi_{ij})$ be an inductive system of discrete quantum groups with inductive limit $\h{\G}$ and limit maps $\pi_{i} : C(\G_{i})\rightarrow C(\G)$. Then, if all the maps $\pi_{i}$ are injective,
\begin{equation*}
\sup_{i}\Lambda_{cb}(\h{\G}_{i}) = \Lambda_{cb}(\h{\G}).
\end{equation*}
In particular, the inductive limit is weakly amenable if and only if the quantum groups are all weakly amenable with uniformly bounded Cowling-Haagerup constant.
\end{prop}

\begin{proof}
The injectivity of the limit maps ensures that each $\h{\G}_{i}$ can be seen as a discrete quantum subgroup of $\G$. Hence, Corollary \ref{cor:wasubgroup} gives the inequality
\begin{equation*}
\sup_{i}\Lambda_{cb}(\h{\G}_{i}) \leqslant \Lambda_{cb}(\h{\G}).
\end{equation*}
To prove the converse inequality, fix an $\epsilon > 0$ and let $(a^{i}_{t})_{t}$ be a net of finitely supported elements in $\ell^{\infty}(\h{\G}_{i})$ converging pointwise to $1$ and such that $\limsup \|m_{a^{i}_{t}}\|_{cb}\leqslant \Lambda_{cb}(\h{\G}_{i}) + \epsilon$. Using the description of the representation theory of inductive limits given by Proposition \ref{prop:inductivelimit}, we can see $(a^{i}_{t})_{(i, t)}$ as a net of finitely supported elements of $\ell^{\infty}(\h{\G})$ converging pointwise to $1$ by setting $a^{i}_{t}p_{\alpha} = 0$ for any $\alpha\notin \Ir(\G_{i})$. The associated multiplier is $m_{a^{i}_{t}}\circ \E_{\G_{i}}$, so that it has the same completely bounded norm as $m_{a_{t}^{i}}$. The conditions on the completely bounded norms then gives $\Lambda_{cb}(\h{\G}) \leqslant \sup_{i}\Lambda_{cb}(\h{\G}_{i}) + \epsilon$.
\end{proof}

\begin{prop}
Let $(\h{\G}_{i}, \pi_{ij})$ be an inductive system of discrete quantum groups with inductive limit $\h{\G}$ and limit maps $\pi_{i} : C(\G_{i})\rightarrow C(\G)$. Then, if all the maps $\pi_{i}$ are injective, $\h{\G}$ has the Haagerup property if and only if all the quantum groups $\G_{i}$ have the Haagerup property.
\end{prop}

\begin{proof}
The "only if" part comes from stability under passing to quantum subgroups. To prove the "if" part, we simply have to prove that complete positivity is preserved when a multiplier is extended to the inductive limit. This comes from the fact that the multiplier is equal to $m_{a_{t}^{i}}\circ\E_{\G_{i}}$.
\end{proof}

\subsection{Amalgamated free products}

Consider two discrete quantum groups $\h{\G}_{1}$ and $\h{\G}_{2}$ together with a common discrete quantum subgroup $\h{\HH}$ and let us consider the C*-algebra $A$ obtained by taking the reduced amalgamated free product $C_{\text{red}}(\G_{1})\ast_{C_{\text{red}}(\HH)} C_{\text{red}}(\G_{2})$ with respect to the conditional expectations given by \ref{prop:conditionalexpectationsubgroup}. The coproducts of $\G_{1}$ and $\G_{2}$ induce a map $\D$ on $A$ which is shown in \cite{wang1995free} to be a coproduct turning $(A, \D)$ into a compact quantum group. In analogy with the classical case, the dual of $A$ will be called the \emph{free product of $\h{\G}_{1}$ and $\h{\G}_{2}$ amalgamated over $\h{\HH}$}.

It is well-known that a free product of amenable groups need not be amenable. However, it was proved in \cite[Thm 7.8]{daws2014haagerup} that the Haagerup property passes to free products of discrete quantum groups and it was proved in \cite{freslon2012note} that a free product of discrete quantum groups with Cowling-Haagerup constant equal to $1$ again has Cowling-Haagerup constant equal to $1$. Our goal in this subsection is to extend those two results by allowing amalgamation over a finite quantum subgroup. Note that such a statement for the Haagerup property was proved in \cite[Prop 7.13]{daws2014haagerup} when the quantum groups are unimodular, using the associated von Neumann algebras. Our proof does not require unimodularity, but the price to pay is dealing all the way long with multipliers. For this, we need the following generalization of Gilbert's criterion proved in \cite[Prop 4.1 and Thm 4.2]{daws2012completely}.

\begin{thm}[Daws]\label{thm:quantumgilbert}
Let $\h{\G}$ be a discrete quantum group and let $a\in \ell^{\infty}(\h{\G})$. Then, $m_{a}$ extends to a completely bounded multiplier on $\B(L^{2}(\G))$ if and only if there exists a Hilbert space $K$ and two maps $\xi, \eta \in \B(L^{2}(\G), L^{2}(\G)\otimes K)$ such that $\|\xi\|\|\eta\| = \|m_{a}\|_{cb}$ and
\begin{equation}\label{eq:quantumgilbert}
(1\otimes \eta)^{*}\h{W}_{12}^{*}(1\otimes \xi)\h{W} = a\otimes 1.
\end{equation}
Moreover, we then have $m_{a}(x) = \eta^{*}(x\otimes 1)\xi$.
\end{thm}

Let us give a proof of the stability of the Haagerup property under free products in the language of multipliers, in order to make the extension to the amalgamated case more clear.

\begin{prop}\label{prop:haagerupnonamalgamates}
Let $(\h{\G}_{i})_{i\in I}$ be a family of discrete quantum groups with the Haagerup property. Then, $\ast_{i}\h{\G}_{i}$ has the Haagerup property.
\end{prop}

\begin{proof}
It is clearly enough to prove the result for two quantum groups $\G_{1}$ and $\G_{2}$. First note that according to \cite[Thm 5.9]{daws2012completely}, the complete positivity of a multiplier $m_{a}$ implies that $a\in C_{0}(\h{\G}_{i})$ can in fact be chosen to be of the form $(\w_{a}\otimes \i)(W_{i})$ for some state $\w_{a}$ on the envelopping C*-algebra $C_{\text{max}}(\G_{i})$ of $\Pol(\G_{i})$. This means that if we consider two elements $a\in C_{0}(\h{\G}_{1})$ and $b\in C_{0}(\h{\G}_{2})$, then the free product $m_{a}\ast m_{b}$ of completely positive maps coresponds to the multiplier $m_{c}$ with
\begin{equation*}
c = (\w_{c}\otimes \i)(W),
\end{equation*}
where $\w_{c}$ is the free product of the states $\w_{a}$ and $\w_{b}$. So let us take nets $(\w_{a_{t}}\otimes \i)(W_{1})$ and $(\w_{b_{t}}\otimes \i)(W_{2})$ implementing the Haagerup property for $\h{\G}_{1}$ and $\h{\G}_{2}$ respectively. Using the remark above, we get a net of elements $c_{t} = (\w_{c_{t}}\otimes \i)(W)$ converging pointwise to the identity and yielding completely positive multipliers (because $\w_{c_{t}}$ is again a state). Now, \cite[Thm 3.9]{boca1993method} asserts that the free product map $m_{c_{t}} = m_{a_{t}}\ast m_{b_{t}}$ is $L^{2}$-compact as soon as $m_{a_{t}}$ and $m_{b_{t}}$ are. If $\h{\K} = \h{\G}_{1}\ast\h{\G}_{2}$, we have for any $\alpha\in \Ir(\K)$,
\begin{eqnarray*}
(h\otimes \i)((m_{c_{t}}\otimes \i)(u^{\alpha})(u^{\alpha})^{*}) & = & 
(h\otimes \i)((\w_{c_{t}}\otimes \i\otimes \i)(u^{\alpha}_{13}u^{\alpha}_{23})(u^{\alpha})^{*}) \\
& = & (h\otimes \i)((\w_{c_{t}}\otimes \i\otimes \i)[u^{\alpha}_{13}u^{\alpha}_{23}(u^{\alpha}_{23})^{*}]) \\
& = & (h\otimes \i)((\w_{c_{t}}\otimes \i\otimes \i)(u^{\alpha}_{13})) \\
& = & c_{t}p_{\alpha}.
\end{eqnarray*}
This and the fact that $m_{c_{t}}$ is $L^{2}$-compact prove that $c_{t}\in C_{0}(\h{\G}_{1}\ast\h{\G}_{2})$, concluding the proof.
\end{proof}

The strategy to handle the amalgamated case is to produce multipliers which, while implementing the desired approximation property, are the identity on the amalgam. Here finiteness proves crucial through the next Lemma. If $\h{\HH}$ is a finite quantum group, it is unimodular and we will denote by $\h{h}$ the unique Haar weight on $\ell^{\infty}(\h{\HH})$ which is both left and right invariant, i.e. $\h{h}(a) = \sum_{\alpha}\Tr(ap_{\alpha})$. We first define the averaging process.

\begin{de}
Let $\h{\G}$ be a discrete quantum group, $\h{\HH}$ a \emph{finite} quantum subgroup of $\h{\G}$ and let $a\in \ell^{\infty}(\h{\G})$. The \emph{averaging of $a$ over $\h{\HH}$} is the element $c\in \ell^{\infty}(\h{\G})$ defined by
\begin{equation*}
c = (\h{h}\otimes \i)[(p_{\HH}\otimes \i)\h{\D}(a)].
\end{equation*}
\end{de}

We now prove that this averaging process is well-behaved with respect to completely bounded norms and that it yields the identity on $\h{\HH}$.

\begin{lem}\label{lem:averagingmultipliers}
The averaging of $a$ over $\h{\HH}$ satisfies $\|m_{c}\|_{cb}\leqslant \vert \h{h}(p_{\HH})\vert \|m_{a}\|_{cb}$. Moreover, $m_{c}$ is a multiple of the identity on $C_{\text{red}}(\HH)\subset C_{\text{red}}(\G)$.
\end{lem}

\begin{proof}
If $\eta, \xi : L^{2}(\G) \rightarrow L^{2}(\G)\otimes K$ are the maps coming from Theorem \ref{thm:quantumgilbert}, we set
\begin{equation}
\xi' = (\h{h}\otimes \i\otimes \i)[\h{W}_{12}^{*}(p_{\HH}\otimes \xi)\h{W}].
\end{equation}
Let us make the meaning of this definition clear : if $x\in \ell^{2}(\G)\otimes L^{2}(\G)$, then $\h{W}x$ belongs to the same space so that we can apply $\xi$ to its second leg, yielding an element $y\in \ell^{2}(\h{\G})\otimes L^{2}(\G)\otimes K$ to which we can apply $\h{W}^{*}_{12}$. Now, taking $p_{\HH}$ into account we see that the term inside the brackets is a linear map from $\ell^{2}(\HH)\otimes L^{2}(\G)$ to $\ell^{2}(\HH)\otimes L^{2}(\G)\otimes K$. Eventually, applying $\h{h}$ to the first leg (note that it is defined on all $\B(\ell^{2}(\h{\HH}))$ because it is finite-dimensional) gives a map from $L^{2}(\G)$ to $L^{2}(\G)\otimes K$.

We are going to prove that $c\otimes 1 = (1\otimes \eta^{*})\h{W}_{12}^{*}(1\otimes \xi')\h{W}$, which will imply our claim on the completely bounded norm since $\|\xi'\|\leqslant \vert \h{h}(p_{\HH})\vert \|\xi\|$. First note that since
\begin{equation*}
(\h{\D}\otimes \i)(x) = \h{W}_{12}^{*}(1\otimes x)\h{W}_{12},
\end{equation*}
we have
\begin{eqnarray*}
\h{\D}(a)\otimes 1 & = & \h{W}_{12}^{*}(1\otimes 1\otimes \eta^{*})(1\otimes \h{W}^{*}_{12})(1\otimes 1\otimes \xi)(1\otimes \h{W})\h{W}_{12} \\
& = & (1\otimes 1\otimes \eta^{*})(\h{W}^{*}_{12}\otimes 1)(1\otimes \h{W}_{12}^{*})(\h{W}_{12}\otimes 1)(1\otimes 1\otimes \xi)\h{W}_{12}^{*}\h{W}_{23}\h{W}_{12} \\
& = & (1\otimes 1\otimes \eta^{*})(\h{W}^{*}_{23}\h{W}_{13}^{*}\otimes 1)(1\otimes 1\otimes \xi)\h{W}_{13}\h{W}_{23}
\end{eqnarray*}
where we used twice the pentagonal equation for $\h{W}$. Applying $\h{h}(p_{\HH}\: .)$ to the first leg yields the result. Now, if $\alpha\in \Ir(\HH)$, we get using the invariance of $\h{h}$,
\begin{eqnarray*}
cp_{\alpha} & = & (\h{h}\otimes \i)[(p_{\HH}\otimes \i)\h{\D}(a)]p_{\alpha} \\
& = & (\h{h}\otimes \i)[(p_{\HH}\otimes p_{\HH})\h{\D}(a)]p_{\alpha} \\
& = & (\h{h}\otimes \i)[\h{\D}(p_{\HH}a)]p_{\alpha} \\
& = & \h{h}(p_{\HH}a)p_{\alpha}
\end{eqnarray*}
and $m_{c} = \h{h}(p_{\HH}a)\Id$ on $C_{\text{red}}(\HH)$.
\end{proof}

Using this averaging technique, we get the result for the Haagerup property.

\begin{thm}\label{thm:haagerupamalgamated}
Let $(\h{\G}_{i})_{i\in I}$ be a family of discrete quantum groups with the Haagerup property and let $\h{\HH}$ be a common \emph{finite} quantum subgroup. Then, $\ast_{\h{\HH}}\h{\G}_{i}$ has the Haagerup property.
\end{thm}

\begin{proof}
We first prove that completely positive multipliers on $C_{\text{red}}(\G_{i})$ can be averaged so that they become the identity on $C_{\text{red}}(\HH)$. Let us first recall that by \cite[Prop 2.6]{kraus1999approximation}, if $c\in \ell^{\infty}(\h{\G}_{i})$ satisfies $m_{c}(x) = \eta^{*}(x\otimes 1)\xi$ for some maps $\xi, \eta\in \B(L^{2}(\G_{i}), L^{2}(\G_{i})\otimes K)$, then
\begin{equation*}
m_{\h{S}(c)^{*}}(x) = \xi^{*}(x\otimes 1)\eta,
\end{equation*}
where $\h{S}$ denotes the antipode of $\h{\G}_{i}$. Let $(a_{t})$ be a net of elements in $C_{0}(\h{\G}_{i})$ converging pointwise to $1$ and such that $m_{a_{t}}$ is unital and completely positive. Let $b_{t}$ be the averaging of $a_{t}$ over $\h{\HH}$. Average again $\h{S}(b_{t})^{*}$ to produce a third element $b'_{t}$. Then, by Lemma \ref{lem:averagingmultipliers}, $(b'_{t})$ is a net of elements in $C_{0}(\h{\G}_{i})$ converging pointwise to $1$ and such that $m_{b'_{t}}$ is the identity on $C_{\text{red}}(\HH)$. In the case of a unital completely positive multiplier, the proof of Theorem \ref{thm:quantumgilbert} yields a map $\xi : L^{2}(\G_{i}) \rightarrow L^{2}(\G_{i})\otimes K$ such that $m_{a_{t}}(x) = \xi^{*}(x\otimes \i)\xi$. This implies that $m_{b'_{t}}(x) = (\xi')^{*}(x \otimes 1)\xi'$ is unital and completely positive. Moreover, the conditional expectation $\E_{\HH}\circ m_{b'_{t}}$ is invariant with respect to the Haar state, because $m_{b'_{t}}$ is invariant. Since, by \cite[Prop 2.2]{vergnioux2004k}, $\E_{\HH}$ is the unique conditional expectation satisfying this invariance, we have $\E_{\HH}\circ m_{b'_{t}} = \E_{\HH}$. 

Now that all the multipliers are the identity on $C_{\text{red}}(\HH)$ and preserve the conditional expectation $\E_{\HH}$, their amalgamated free product makes sense (see for example \cite[Thm 4.8.5]{brown2008finite} for details on the construction of the free product of u.c.p. maps). We can then follow the proof of Proposition \ref{prop:haagerupnonamalgamates} to conclude.
\end{proof}

Let us turn to weak amenability, which is more involved, and let us check that our averaging technique preserves weak amenability.

\begin{lem}\label{lem:invariantwa}
Let $\h{\G}$ be a discrete quantum group, $\h{\HH}$ a \emph{finite} quantum subgroup of $\h{\G}$ and let $(a_{t})$ be a net of finitely supported elements in $\ell^{\infty}(\h{\G})$ converging pointwise to $1$. Then, there exists a net $(b_{t})$ of finite-rank elements in $\ell^{\infty}(\h{\G})$ converging pointwise to $1$, such that
\begin{equation*}
\limsup_{t} \|m_{b_{t}}\|_{cb}\leqslant \limsup_{t}\|m_{a_{t}}\|_{cb}
\end{equation*}
and $m_{b_{t}}$ is the identity on $C_{\text{red}}(\HH)\subset C_{\text{red}}(\G)$.
\end{lem}

\begin{proof}
Let $c_{t}$ be the averaging of $a_{t}$ over $\h{\HH}$ and let $\Supp(a_{t})$ be the set of equivalence classes of irreducible representations $\alpha$ of $\G$ such that $a_{t}p_{\alpha}\neq 0$. Then,
\begin{equation*}
c_{t}p_{\alpha} = (\h{h}\otimes \i)\left[(p_{\HH}\otimes p_{\alpha})\h{\D}(a)\right] = (\h{h}\otimes \i)\left[\sum_{\beta\in \Ir(\HH)}\sum_{\gamma\in \Supp(a)}\h{\D}(p_{\gamma})(p_{\beta}\otimes p_{\alpha})\h{\D}(a)\right]
\end{equation*}
By definition, $\h{\D}(p_{\gamma})(p_{\beta}\otimes p_{\alpha})\neq 0$ if and only if $\gamma\subset \beta\otimes \alpha$, which by Frobenius reciprocity (see e.g. \cite[Prop 3.1.11]{timmermann2008invitation}) is equivalent to $\overline{\alpha}\subset \overline{\gamma}\otimes \beta$. Hence, $c_{t}p_{\alpha}$ is non-zero only if $\alpha$ belongs to the finite set
\begin{equation*}
\bigcup_{\beta\in \Ir(\HH)}\bigcup_{\gamma\in \Supp(a)}\{\alpha\in \Ir(\G), \overline{\alpha}\subset \overline{\gamma}\otimes \beta\}
\end{equation*}
and $c_{t}$ has finite support. The same holds for
\begin{equation*}
b_{t} = \frac{1}{\h{h}(p_{\HH}a_{t})}c_{t},
\end{equation*}
which induces multipliers $m_{b_{t}}$ which are the identity on $C_{\text{red}}(\HH)$.

Assume now that $a_{t}$ converges pointwise to $1$, and note that the inequality of the completely bounded norms is obvious from Lemma \ref{lem:averagingmultipliers}. Fix $\beta\in \Ir(\G)$, $\epsilon > 0$ and let $D$ be the (finite) set of $\gamma\in \Ir(\G)$ which are contained in $\alpha\otimes \beta$ for some $\alpha\in \Ir(\HH)$. We denote by $p_{D}$ the sum of the projections $p_{\gamma}$ for $\gamma\in D$. Let $t$ be such that :
\begin{itemize}
\item $\|a_{t}p_{D} - p_{D}\|\leqslant \displaystyle\frac{\epsilon}{4}$
\item $\|a_{t}p_{\HH} - p_{\HH}\|\leqslant \displaystyle\frac{\epsilon}{4}$
\item $\vert \h{h}(a_{t}p_{\HH})\vert > \displaystyle\frac{1}{2}$
\end{itemize}
Then,
\begin{eqnarray*}
\|b_{t}p_{\beta} - p_{\beta}\| & = & \|\vert\h{h}(a_{t}p_{\HH})\vert^{-1}(\h{h}\otimes \i)[(p_{\HH}\otimes p_{\beta})\h{\D}(a_{t})] - p_{\beta}\| \\
& = & \vert\h{h}(a_{t}p_{\HH})\vert^{-1}\|(\h{h}\otimes \i)[(p_{\HH}\otimes p_{\beta})\h{\D}(a_{t}) - (a_{t}p_{\HH}\otimes p_{\beta})]\| \\
& = & \vert\h{h}(a_{t}p_{\HH})\vert^{-1}\|(\h{h}\otimes \i)[(p_{\HH}\otimes p_{\beta})\h{\D}(a_{t}p_{D}) - (a_{t}p_{\HH}\otimes p_{\beta})]\| \\
& \leqslant & \vert\h{h}(a_{t}p_{\HH})\vert^{-1}\|(\h{h}\otimes \i)[(p_{\HH}\otimes p_{\beta})\h{\D}(a_{t}p_{D} - p_{D})]\| \\
& + & \vert\h{h}(a_{t}p_{\HH})\vert^{-1}\|(\h{h}\otimes \i)[(p_{\HH}\otimes p_{\beta})\h{\D}(p_{D}) - p_{\HH}\otimes p_{\beta}]\| \\
& + & \vert\h{h}(a_{t}p_{\HH})\vert^{-1}\|(\h{h}\otimes \i)[(p_{\HH} - a_{t}p_{\HH})\otimes p_{\beta})]\| \\
& \leqslant & \vert \h{h}(a_{t}p_{\HH})\vert^{-1}(\|a_{t}p_{D} - p_{D}\| + \|a_{t}p_{\HH} - p_{\HH}\|) \\
& \leqslant & 2\left(\frac{\epsilon}{4} + \frac{\epsilon}{4}\right) = \epsilon
\end{eqnarray*}
using the fact that $(p_{\HH}\otimes p_{\beta})\h{\D}(p_{D}) = p_{\HH}\otimes p_{\beta}$.
\end{proof}

From this it is easy to prove using the amalgamated version of \cite[Thm 4.3]{ricard2005khintchine} that if the quantum groups $\G_{i}$ are amenable, then the free product amalgamated over a finite quantum subgroup is weakly amenable with Cowling-Haagerup constant equal to $1$, thus generalizing a result of M. Bo\.{z}ejko and M.A. Picardello \cite{bozejko1993weakly} (see \cite[Thm 2.3.15]{freslon2013proprietes} for a proof). When the quantum groups are not amenable but are weakly amenable with Cowling-Haagerup constant equal to $1$, we need the full power of the work of E. Ricard and X. Qu. Let us introduce some notations : if $A$ is a C*-algebra with a conditional expectation $\E$, we denote by $L^{2}(A, \E)$ (resp. $L^{2}(A, \E)^{op}$) the Hilbert module obtained by the GNS constructions using the inner product $\langle a, b\rangle = \E(ab^{*})$ (resp. $\langle a, b\rangle = \E(a^{*}b)$).

\begin{thm}[Ricard, Xu]\label{thm:freecmap}
Let $C$ be a C*-algebra, let $(B_{i})_{i\in I}$ be unital C*-algebras together with GNS-faithful conditional expectations $\E_{B_{i}} : B_{i} \rightarrow C$. Let $A_{i}\subset B_{i}$ be unital C*-subalgebras with GNS-faithful conditional expectations $\E_{A_{i}} : A_{i} \rightarrow C$ which are the restrictions of $\E_{B_{i}}$. Assume that for each $i$, there is a net of finite-rank maps $(V_{i, j})_{j}$ on $A_{i}$ converging to the identity pointwise, satisfying $\E_{A_{i}}\circ V_{i, j} = \E_{A_{i}}$ and such that $\limsup_{j}\|V_{i, j}\|_{cb}=1$. Assume moreover that for each pair $(i, j)$, there is a completely positive unital map $U_{i, j} : A_{i} \rightarrow B_{i}$ satisfying $\E_{B_{i}}\circ U_{i, j} = \E_{A_{i}}$ and such that
\begin{equation*}
\| V_{i, j} - U_{i, j}\|_{cb} + \| V_{i, j}-U_{i, j}\|_{\B(L^{2}(A_{i}, \E_{A_{i}}), L^{2}(B_{i}, \E_{B_{i}}))} + \| V_{i, j} - U_{i, j}\|_{\B(L^{2}(A_{i}, \E_{A_{i}})^{op}, L^{2}(B_{i}, \E_{B_{i}})^{op})} \underset{j}{\rightarrow} 0.
\end{equation*}
Assume moreover that the maps $V_{i, j}$ and $U_{i, j}$ are the identity on $C$ for all $i, j$. Then, the reduced amalgamated free product $\ast_{C} (A_{i}, \E_{A_{i}})$ has Cowling-Haagerup constant equal to $1$.
\end{thm}

\begin{rem}
This statement is the same as \cite[Prop 4.11]{ricard2005khintchine} with states replaced by conditional expectations. This amalgamated version does not appear explicitly in \cite{ricard2005khintchine} but is a straightforward consequence of the amalgamated version of the Khintchine inequality proved in \cite[Sec 5]{ricard2005khintchine}. The idea of the proof is that the "free product" of the maps $V_{i, j}$, which does not make sense \emph{a priori} can be defined using the free product of the u.c.p. maps $U_{i, j}$ and the approximation assumption. Let us note that it is necessary to have bigger C*-algebras $B_{i}$ as ranges for the completely positive maps $U_{i, j}$, otherwise we would be assuming that the C*-algebras $A_{i}$ are nuclear.
\end{rem}

Based on the non-amalgamated case \cite{freslon2012note}, one can try the following strategy : if $(a_{t})$ is a net of elements implementing weak amenability, first average $a_{t}$ and then $S(a_{t})^{*}$ over $\h{\HH}$ to produce an element $b_{t}'$ in $\ell^{\infty}(\h{\G})$ and two maps $\eta'_{t}$ and $\xi'_{t}$ in $\B(L^{2}(\G), L^{2}(\G)\otimes K)$ satisfying
\begin{equation*}
m_{b_{t}'}(x) = (\eta'_{t})^{*}(x\otimes 1)\xi'_{t}
\end{equation*}
and such that the multiplier $m_{b_{t}'}$ is the identity on $C_{\text{red}}(\HH)$. Note that $\|1 - (\gamma'_{t})^{*}\gamma'_{t}\|\leqslant \|m_{b'_{t}}\|_{cb} -1$ so that $(\gamma'_{t})^{*}\gamma'_{t}$ is invertible if $m_{b'_{t}}$ is sufficiently close to the identity in completely bounded norm. Then, mimicking \cite[Lem 4.3]{freslon2012note}, setting $\gamma'_{t} = (\eta'_{t}+\xi'_{t})/2$ and $\tilde{\gamma}'_{t} = \gamma'_{t}\vert\gamma'_{t}\vert^{-1}$, the unital completely positive approximation we are looking for should be
\begin{equation*}
M_{\tilde{\gamma}'_{t}}(x) = (\tilde{\gamma}'_{t})^{*}(x\otimes 1)\tilde{\gamma}'_{t}.
\end{equation*}
The problem is then to prove that this operator is the identity on $C_{\text{red}}(\HH)$. We do not know whether this fact holds or not. However, we can use again an averaging trick, this time at the level of C*-algebras, to build a new unital completely positive approximation which will be the identity on $C_{\text{red}}(\HH)$. If $T : C_{\text{red}}(\G)\rightarrow \B(L^{2}(\G))$ is a linear map, we define a linear map $R_{\h{\HH}}(T)$ by
\begin{equation*}
R_{\h{\HH}}(T) : x \mapsto \int_{\mathcal{U}(C_{\text{red}}(\HH))} T(xv^{*})v dv,
\end{equation*}
where the integration is done with respect to the normalized Haar measure of the compact group $\mathcal{U}(C_{\text{red}}(\HH))$ (recall that $C_{\text{red}}(\HH)$ is finite-dimensional). Similarly, we define a linear map $L_{\h{\HH}}(T)$ by
\begin{equation*}
L_{\h{\HH}}(T) : x \mapsto \int_{\mathcal{U}(C_{\text{red}}(\HH))} u^{*}T(ux) du.
\end{equation*}

Let us give some elementary properties of these two operations.

\begin{lem}\label{lem:bimodular}
If $T$ is completely bounded, then $R_{\h{\HH}}(T)$ (resp. $L_{\h{\HH}}(T)$) is also completely bounded with $\|R_{\h{\HH}}(T)\|_{cb}\leqslant \|T\|_{cb}$ (resp. $\|L_{\h{\HH}}(T)\|_{cb} \leqslant \|T\|_{cb}$). Moreover, for any $a, b\in C_{\text{red}}(\HH)$,
\begin{equation}\label{eq:bimodular}
R_{\h{\HH}}\circ L_{\h{\HH}}(T)(axb) = a [R_{\h{\HH}}\circ L_{\h{\HH}}(T)(x)] b.
\end{equation}
\end{lem}

\begin{proof}
Let us prove the first part for $R_{\h{\HH}}(T)$ (the computation is similar for $L_{\h{\HH}}(T)$). For any integer $n$ and any $x\in C_{\text{red}}(\G)\otimes M_{n}(\C)$, we have
\begin{eqnarray*}
\|(R_{\h{\HH}}(T)\otimes \Id_{M_{n}(\C)})(x)\| & \leqslant & \int_{\mathcal{U}(C_{\text{red}}(\HH))} \|(T\otimes \Id_{M_{n}(\C)})(x(v^{*}\otimes 1))(v\otimes 1)\| dv \\
& \leqslant & \int_{\mathcal{U}(C_{\text{red}}(\HH))} \|T\|_{cb}\|x\| dv \\
& = & \|T\|_{cb}\|x\|.
\end{eqnarray*}
Writing any element in $C_{\text{red}}(\HH)$ as a linear combination of four unitaries (up to a scalar), we can restrict ourselves to prove Equation (\ref{eq:bimodular}) when $a$ and $b$ are unitaries. In that case, the changes of variables $u = u'a$ and $v = v'b$ yield
\begin{eqnarray*}
R_{\h{\HH}}\circ L_{\h{\HH}}(T)(a^{*}xb) & = & \iint_{\mathcal{U}(C_{\text{red}}(\HH))\times \mathcal{U}(C_{\text{red}}(\HH))} u^{*}T(ua^{*}xbv^{*})v du dv \\
& = & \iint_{\mathcal{U}(C_{\text{red}}(\HH))\times \mathcal{U}(C_{\text{red}}(\HH))} a^{*}(u')^{*}T(u'x(v')^{*})v'b du' dv' \\
& = & a^{*} [R_{\h{\HH}}\circ L_{\h{\HH}}(T)(x)] b.
\end{eqnarray*}
\end{proof}

With this in hand, we will be able to average the completely positive maps approximating the multipliers. Let us check that this averaging behaves nicely on the multipliers $m_{b'(t)}$.

\begin{lem}\label{lem:averagingcb}
The maps $A_{t} = R_{\h{\HH}}\circ L_{\h{\HH}}(m_{b'(t)})$ have finite rank, converge pointwise to the identity, are equal to the identity on $C_{\text{red}}(\HH)$ and satisfy $\limsup_{t}\|A_{t}\|_{cb} = \limsup_{t}\|m_{b'(t)}\|_{cb}$.
\end{lem}

\begin{proof}
The pointwise convergence, the identity property and the bound on the completely bounded norms follow from the construction and Lemma \ref{lem:bimodular}. To prove that the rank is finite, first note that if $\alpha\in \Ir(\G)$ and $u, v\in C_{\text{red}}(\HH)$, then $u(u^{\alpha}_{i, j})v$ belongs to the linear span of coefficients of irreducible subrepresentations $\gamma$ of $\beta_{1}\otimes \alpha\otimes \beta_{2}$ for $\beta_{1}, \beta_{2}\in \Ir(\HH)$. Thus, by Frobenius reciprocity,
\begin{equation*}
A_{t}(u^{\alpha}_{i, j}) = \iint_{\mathcal{U}(C_{\text{red}}(\HH))\times \mathcal{U}(C_{\text{red}}(\HH))} u^{*}m_{b'(t)}(u(u^{\alpha}_{i, j})v^{*})v du dv
\end{equation*}
is equal to $0$ as soon as $\alpha$ is not in the finite the set
\begin{equation*}
\bigcup_{\beta_{1}, \beta_{2}\in \Ir(\HH)}\bigcup_{\gamma\in\Supp(b'(t))}\{\alpha\in \Ir(\G), \overline{\alpha}\in \beta_{2}\otimes \overline{\gamma}\otimes \beta_{1}\}.
\end{equation*}
\end{proof}

The following theorem is the the best known statement on stability of weak amenability with respect to free products for discrete quantum groups. Note that, to our knowledge, it is also new for classical groups, even though it may be well-known to experts.

\begin{thm}\label{thm:quantumwaamalgamated}
Let $(\h{\G}_{i})_{i\in I}$ be a family of weakly amenable discrete quantum groups such that $\Lambda_{cb}(\h{\G}_{i}) = 1$ for every $i\in I$ and let $\h{\HH}$ be a common \emph{finite} quantum subgroup. Then,
\begin{equation*}
\Lambda_{cb}(\ast_{\h{\HH}}\h{\G}_{i})=1.
\end{equation*}
\end{thm}

\begin{proof}
Using the notations of Theorem \ref{thm:freecmap}, we set :
\begin{itemize}
\item $A_{i} = C_{\text{red}}(\G_{i})$
\item $B_{i} = \B(L^{2}(\G_{i}))$
\item $C = C_{\text{red}}(\HH)$
\item $\E_{A_{i}} = \E_{\HH}$
\end{itemize}
To define the conditional expectations $\E_{B_{i}}$, first consider the orthogonal projection
\begin{equation*}
P_{\HH}^{i} : L^{2}(\G_{i}) \rightarrow L^{2}(\HH).
\end{equation*}
Then, $\E_{i}' : x\mapsto P_{\HH}^{i}xP_{\HH}^{i}$ is a conditional expectation from $\B(L^{2}(\G_{i}))$ to $\B(L^{2}(\HH))$ with the property that for any coefficient $x$ of an irreducible representation in $\Ir(\G_{i})\setminus\Ir(\HH)$, $\E_{i}'(x) = 0$. In fact, the restriction of $\E_{i}'$ to $C_{\text{red}}(\G_{i})$ is precisely the conditional expectation $\E_{\HH}$ of Proposition \ref{prop:conditionalexpectationsubgroup}. Because $C_{\text{red}}(\HH)$ is finite-dimensional, there is also a conditional expectation
\begin{equation*}
\E''_{i} : \B(L^{2}(\HH)) \rightarrow C_{\text{red}}(\HH).
\end{equation*}
We set $\E_{B_{i}} = \E''_{i}\circ \E'_{i}$. Since $\E_{A_{i}}$ is the restriction of $\E_{i}'$to $A_{i}$ it is also the restriction of $\E_{B_{i}}$.

Let us fix an index $i$, let $(a_{j})_{j}$ be a net of elements in $\ell^{\infty}(\G_{i})$ implementing weak amenability. Applying the construction of Lemma \ref{lem:invariantwa} to $a_{j}$ we get an element $b_{j}$. Applying it again to $S(b_{j})^{*}$ yields another element $b_{j}'$ and we can assume that $b_{j}' = S(b_{j}')^{*}$ as in the proof of \cite[Thm 4.2]{freslon2012note}. We therefore set $V_{i, j} = R_{\h{\HH}}\circ L_{\h{\HH}}(m_{b'_{j}})$. By Lemma \ref{lem:averagingcb}, these maps satisfy all the required properties. Recall that $\gamma'_{j} = (\xi_{j}'+\eta_{j}')/2$ and $\widetilde{\gamma}'_{j} = \gamma'_{j}\vert \gamma'_{j}\vert^{-1}$. Set
\begin{equation*}
\zeta_{j} = \int_{\mathcal{U}(C_{\text{red}}(\HH))} (1\otimes u^{*})\tilde{\gamma}_{j}u du
\end{equation*}
and observe that
\begin{equation*}
R_{\h{\HH}}\circ L_{\h{\HH}}(M_{\tilde{\gamma}_{j}}) = M_{\zeta_{j}} : x\mapsto \zeta^{*}_{j}(1\otimes x)\zeta_{j}
\end{equation*}
is a completely positive map which is the identity on $C_{\text{red}}(\HH)$ by Equation (\ref{eq:bimodular}). We therefore set $U_{i, j} = M_{\zeta_{j}}$. Then, Lemma \ref{lem:averagingcb} yields
\begin{equation*}
\|V_{i, j} - U_{i, j}\|_{cb} =  \|R_{\h{\HH}}\circ L_{\h{\HH}}(m_{b'_{j}} - M_{\tilde{\gamma}_{j}})\|_{cb} \leqslant \|m_{b'_{j}} - M_{\tilde{\gamma}_{j}}\|_{cb},
\end{equation*}
so that, by \cite[Lem 4.3]{freslon2012note} the convergence in completely bounded norm holds. We still have to check the compatibility of the maps with the conditional expectations and the $L^{2}$-norm convergence. 

Let us consider the conditional expectation
\begin{equation*}
\E = \E_{B_{i}}\circ U_{i, j} : C_{\text{red}}(\G_{i}) \rightarrow C_{\text{red}}(\HH).
\end{equation*}
We claim that $\E(x) = 0$ whenever $x$ is a coefficient of a representation in $\Ir(\G)\setminus\Ir(\HH)$. In fact, it follows directly from the explicit expression of $\xi'_{j}$ given in Lemma \ref{lem:averagingmultipliers} that there is a vector $v_{\xi'_{j}}\in K$ such that for any $y\in C_{\text{red}}(\HH)$, $\xi'_{j}(y) = y\otimes v_{\xi'_{j}}$. The same holds for $\eta'_{j}$ and $\gamma'_{j}$ (with different vectors), hence also for $\widetilde{\gamma}'_{j}$. From this, a straightforward calculation yields
\begin{equation}\label{eq:conditionalexpectation}
P_{\HH}M_{\widetilde{\gamma}'_{j}}(x)P_{\HH} = 0.
\end{equation}
Since elements in $\mathcal{U}(C_{\text{red}}(\HH))$ respect the decomposition $L^{2}(\G) = L^{2}(\HH)\oplus L^{2}(\HH)^{\perp}$, Equation \eqref{eq:conditionalexpectation} also holds for $M_{\zeta_{j}}$, proving the claim. Since $\E(x) = x$ for any $x\in C_{\text{red}}(\HH)$, we see that $\E = \E_{\HH}$. The same argument shows that $\E_{A_{i}}\circ V_{i, j} = \E_{A_{i}}$.

Using this compatibility, the same argument as in \cite[Lem 4.5]{freslon2012note} shows that $M_{\widetilde{\gamma}_{j}'}$ also approximates $m_{b_{j}'}$ in both $L^{2}$-norms (here $L^{2}$-norms means norms as operator between the Hilbert modules associated to the conditional expectations). The proof of Lemma \ref{lem:bimodular} also works for the $L^{2}$-norms so that we can conclude that $U_{i, j}$ approximates $V_{i, j}$ in both $L^{2}$-norms, concluding the proof.
\end{proof}

In the unimodular case, one can use a free product trick to deduce results on HNN extensions (as defined in \cite{fima2012k}) from what has been done. The optimal result for unimodular quantum groups, mixing amalgamated free products and HNN extensions, was stated in the language of graphs of quantum groups in \cite[Thm 4.6]{fima2013graphs}. 

It is known that $\Z^{2}\rtimes SL(2, \Z)$ neither has the Haagerup property nor is weakly amenable, whereas both $\Z^{2}$ and $SL(2, \Z)$ do. Thus, these properties are not preserved under extensions in general. Moreover, recall that since an extension of amenable groups is again amenable, the groups $\Z^{2}\rtimes \Z/4\Z$ and $\Z^{2}\rtimes \Z/6\Z$ are amenable. Thus, they have the Haagerup property and are weakly amenable with Cowling-Haagerup constant equal to $1$. However, noticing that
\begin{equation*}
\Z^{2}\rtimes SL(2, \Z) = (\Z^{2}\rtimes \Z/4\Z)\underset{\Z^{2}\rtimes \Z/2\Z}{\ast}(\Z^{2}\rtimes \Z/6\Z),
\end{equation*}
we see that the finiteness condition in Theorem \ref{thm:quantumwaamalgamated} cannot be removed.

\section{Relative amenability}\label{sec:relative}

We now turn to the study of the notion of relative amenability in the context of discrete quantum groups. The definition is straightforward, recalling that $\tau$ is the action of $\h{\G}$ on the quotient.

\begin{de}\label{de:relativelyamenable}
Let $\h{\G}$ be discrete quantum group and let $\h{\HH}$ be a discrete quantum subgroup of $\h{\G}$. We say that $\h{\G}$ is \emph{amenable relative to $\h{\HH}$} if the quotient space has an invariant mean for the action $\tau$, i.e. if there exists a state $m$ on $\ell^{\infty}(\h{\G}/\h{\HH})$ such that for all $x\in \ell^{\infty}(\h{\G}/\h{\HH})$,
\begin{equation*}
(\i\otimes m)\circ \tau(x) = m(x).1
\end{equation*} 
\end{de}

\begin{rem}
One could define an action of a discrete (or even locally compact) quantum group on a von Neumann algebra to be \emph{amenable} if such an invariant mean exists. This is the notion of "amenable homogeneous space" introduced by P. Eymard in \cite{eymard1972moyennes}. However, there is another notion of amenable action, due to R.J. Zimmer \cite{zimmer1978amenable}. This notion was generalized to locally compact groups acting on von Neumann algebras by C. Anantharaman-Delaroche in \cite{anantharaman1979action} and to Kac algebras by M. Joita and S. Petrescu in \cite{joita1990amenable}. As one can expect, this notion is dual to ours in the following sense : if $\G$ is a Kac algebra together with an amenable action (in the sense of \cite[Def 3.1]{joita1990amenable}) on a von Neumann algebra $M$ and if $M$ admits an invariant state, then $\G$ is amenable (see \cite[Thm 3.5]{joita1990amenable} for a proof). Let us also mention that a C*-algebraic version of amenable actions of discrete quantum groups in the sense of R.J. Zimmer was introduced by S. Vaes and R. Vergnioux in \cite{vaes2007boundary}.
\end{rem}

\begin{rem}
Relative amenability does not pass to subgroups, even in the classical case. Examples of triples of discrete groups $K < H < G$ with $G$ amenable relative to $K$ but $H$ not amenable relative to $K$ were constructed in \cite{monod2003on} and \cite{pestov2003on}.
\end{rem}

It follows from \cite[Thm 7.8]{bedos2005amenability} that if the quasi-regular representation of $\h{\G}$ modulo $\h{\HH}$ weakly contains the trivial one, then it has a left-invariant mean, implying in turn that $\h{\G}$ is amenable relative to $\h{\HH}$. In the classical case, all three properties are known to be equivalent. Such an equivalence is not known yet for discrete quantum groups, but a weak converse involving correspondences will be given in Proposition \ref{prop:weakcorrespondence}.

\subsection{Amenable equivalence}

We are going to use ideas from the work of C. Anantharaman-Delaroche on group actions on von Neumann algebras \cite{anantharaman1979action} and follow the path of \cite{anantharaman1995amenable} to prove a general statement on von Neumann algebras associated to relatively amenable discrete quantum groups, involving the notion of \emph{amenable equivalence} of von Neumann algebras which was introduced in \cite[Def 4.1]{anantharaman1995amenable}. Recall from \cite{rieffel1974morita} that two von Neumann algebras $M$ and $N$ are said to be \emph{Morita equivalent} if there exists an $M-N$ correspondence $H$ such that $M$ is isomorphic to $\mathcal{L}_{N}(H)$.

\begin{de}\label{de:amenablyequivalent}
Let $M$ and $N$ be two von Neumann algebras. We say that $M$ is \emph{amenably dominated} by $N$ if there exists a von Neumann algebra $N_{1}$ which is Morita equivalent to $N$ and contains $M$ in such a way that there is a norm-one projection from $N_{1}$ to $M$. We then write $M\prec_{a} N$. We say that $M$ and $N$ are \emph{amenably equivalent} if $M\prec_{a} N$ and $N\prec_{a} M$.
\end{de}

The following theorem is a generalization of a classical result (see Paragraphe 4.10 in \cite{anantharaman1995amenable}).

\begin{thm}\label{thm:amenableequivalence}
Let $\h{\G}$ be a discrete quantum group and let $\h{\HH}$ be a discrete quantum subgroup such that $\h{\G}$ is amenable relative to $\h{\HH}$. Then, $L^{\infty}(\HH)$ is amenably equivalent to $L^{\infty}(\G)$.
\end{thm}

In order to prove this theorem, we have to build a norm-one projection from a well-chosen von Neumann algebra to $L^{\infty}(\G)$. We will do this by adapting some ideas of \cite{anantharaman1979action} to the setting of discrete quantum groups.

\begin{lem}\label{lem:relativeamenability}
Let $M_{1}$, $M_{2}$ be von Neumann algebras and let $\rho$ be an action of a discrete quantum group $\h{\G}$ on $M_{2}$. Assume that we have a von Neumann subalgebra $N_{2}$ of $M_{2}$ which is stable under the action $\rho$ and a norm-one (non-necessarily normal) equivariant surjective projection $P : M_{2} \rightarrow N_{2}$ (i.e. $\rho\circ P = (\i\otimes P)\circ\rho$). Then, there exists a norm-one surjective projection $Q : M_{1}\overline{\otimes} M_{2} \rightarrow M_{1}\overline{\otimes} N_{2}$ such that
\begin{equation*}
Q(x_{1}\otimes x_{2}) = x_{1}\otimes P(x_{2})
\end{equation*}
for all $x_{1}\in M_{1}$ and $x_{2}\in M_{2}$. Moreover, $Q$ is equivariant with respect to $\mu = (\sigma\otimes \i)\circ(\i\otimes \rho)$.
\end{lem}

\begin{proof}
The existence of the projection $Q$ is proved in \cite[Thm 4]{tomiyama1969projection} (we thank the referee for pointing out to us this reference). It is proved in \cite[Lem 2.1]{anantharaman1979action} that J. Tomiyama's construction preserves equivariance under a classical group action and the proof for a quantum group action is exactly the same.
\end{proof}

\begin{prop}\label{prop:proj}
Let $\h{\G}$ be a discrete quantum group and let $(M, N, P)$ be a triple consisting in a von Neumann algebra $M$ endowed with an action $\rho$ of $\h{\G}$, a von Neumann subalgebra $N$ of $M$ which is stable under the action $\rho$ and a norm-one equivariant projection $P : M \rightarrow N$. Then, there exists a norm-one projection $\widetilde{Q} : \h{\G}\ltimes M \rightarrow \h{\G}\ltimes N$. 
\end{prop}

\begin{proof}
By Lemma \ref{lem:relativeamenability}, there exists a norm-one projection
\begin{equation*}
Q : \B(L^{2}(\G))\overline{\otimes} M \rightarrow \B(L^{2}(\G))\overline{\otimes} N
\end{equation*}
which is equivariant with respect to $\mu = (\sigma\otimes \i)\circ(\i\otimes \rho)$ and such that
\begin{equation*}
Q(x_{1}\otimes x_{2}) = x_{1}\otimes P(x_{2}).
\end{equation*}
Let us consider the explicit $*$-isomorphisms of \cite[Thm 2.6]{vaes2001unitary}
\begin{equation*}
\begin{array}{ccccc}
\Phi_{M} & : & \B(L^{2}(\G))\overline{\otimes} M & \rightarrow & \G^{op}\ltimes(\h{\G}\ltimes M) \\
\Phi_{N} & : & \B(L^{2}(\G))\overline{\otimes} N & \rightarrow & \G^{op}\ltimes(\h{\G}\ltimes N)
\end{array}
\end{equation*}
Since any von Neumann algebra can be recovered in a crossed-product as the fixed points algebra under the dual action by \cite[Thm 2.7]{vaes2001unitary}, we only have to prove that the norm-one projection
\begin{equation*}
\widetilde{Q} = \Phi_{N}\circ Q\circ \Phi_{M}^{-1} : \G^{op}\ltimes(\h{\G}\ltimes M)\rightarrow \G^{op}\ltimes(\h{\G}\ltimes N)
\end{equation*}
is equivariant with respect to the the bidual action to conclude.

Let us use the notations of \cite{vaes2001unitary}. The bidual action $\h{\h{\rho}}$ on the double crossed product can be transported to $\B(L^{2}(\G))\overline{\otimes} M$ in the following way : there is an operator
\begin{equation*}
\mathcal{J} : L^{\infty}(\G) \rightarrow L^{\infty}(\G'^{\text{ op}})
\end{equation*}
and a map $\gamma = \Ad_{\Sigma V^{*}\Sigma\otimes 1}\circ\mu$, for some operator $V$, such that
\begin{equation*}
\h{\h{\rho}}\circ\Phi = (\mathcal{J}\otimes \Phi)\circ\gamma.
\end{equation*}
It is clear that the $\h{\h{\rho}}$-equivariance of $\widetilde{Q}$ is equivalent to the $\gamma$-equivariance of $Q$. We already know by Lemma \ref{lem:relativeamenability} that
\begin{equation*}
\mu\circ Q = (\i\otimes Q)\circ \mu
\end{equation*}
and, using the approximating projections $P'_{J}$, we see that $\i\otimes Q$ also commutes to $\Ad_{\Sigma V^{*}\Sigma\otimes 1}$. Hence,
\begin{equation*}
\gamma\circ Q = (\i\otimes Q)\circ\gamma.
\end{equation*}
\end{proof}

We can now prove our main result.

\begin{proof}[Proof of Theorem \ref{thm:amenableequivalence}]
The fact that the mean is invariant precisely means that $m$ is an equivariant norm-one projection to the von Neumann subalgebra $\C.1$ of $M$. Consequently, proposition \ref{prop:proj} applied to the triple $(\ell^{\infty}(\h{\G}/\h{\HH}), \C, m)$ yields a norm-one projection from $\h{\G}\ltimes\ell^{\infty}(\h{\G}/\h{\HH})$ to $\h{\G}\ltimes \C = L^{\infty}(\G)$. Since according e.g. to \cite[Rmk 4.3]{vaes2005new}, $\h{\G}\ltimes\ell^{\infty}(\h{\G}/\h{\HH})$ is Morita equivalent to $L^{\infty}(\HH)$, we have proven that $L^{\infty}(\G)\prec_{a}L^{\infty}(\HH)$. The conditional expectation from $L^{\infty}(\G)$ to $L^{\infty}(\HH)$ defined in \cite[Prop 2.2]{vergnioux2004k} then gives $L^{\infty}(\HH)\prec_{a} L^{\infty}(\G)$, concluding the proof. 
\end{proof}

\begin{rem}\label{rem:popa}
The basic construction $\langle L^{\infty}(\G), L^{\infty}(\HH)\rangle$ is naturally isomorphic to $\h{\G}\ltimes\ell^{\infty}(\h{\G}/\h{\HH})$ (see e.g. the beginning of Section $4$ of \cite{vaes2005new}). This means that if $\h{\G}$ is amenable relative to $\h{\HH}$, then $L^{\infty}(\G)$ is amenable relative to $L^{\infty}(\HH)$ in the sense of \cite{monod2003on}.
\end{rem}

Using the machinery of correspondences, we can also give a partial converse to the fact that if $\mathcal{R}$ weakly contains the trivial representations, then $\h{\G}$ is amenable relative to $\h{\HH}$.

\begin{prop}\label{prop:weakcorrespondence}
Let $\h{\G}$ be a discrete quantum group and let $\h{\HH}$ be a discrete quantum subgroup of $\h{\G}$ such that $\h{\G}$ is amenable relative to $\h{\HH}$. Then, the correspondence associated to the quasi-regular representation of $\h{\G}$ modulo $\h{\HH}$ weakly contains the identity correspondence.
\end{prop}

\begin{proof}
Let us denote by $H$ the Hilbert space $L^{2}(\G)$ seen as the standard correspondence between $L^{\infty}(\G)$ and $L^{\infty}(\HH)$. By Theorem \ref{thm:amenableequivalence}, $H$ is a \emph{left injective correspondence} in the sense of \cite[Def 3.1]{anantharaman1995amenable}. By \cite[Prop 3.6]{anantharaman1995amenable}, it is also a \emph{left amenable correspondence} in the sense of \cite[Def 2.1]{anantharaman1995amenable}, i.e. the correspondence $H\otimes \overline{H}$ weakly contains the identity correspondence of $L^{\infty}(\G)$. Moreover, $H\otimes \overline{H}$ is precisely the Hilbert space $\ell^{2}(\h{\G}\ltimes \ell^{\infty}(\h{\G}/\h{\HH}))$ associated to the GNS construction for the dual weight $\tilde{\theta}$ on the crossed-product. Thanks to \cite[Prop 3.10]{vaes2001unitary}, the GNS construction for the dual weight may be explicitly described. In fact, the map
\begin{equation*}
\mathcal{I} : (a\otimes 1)\alpha(x)\mapsto a\otimes x
\end{equation*}
for $a\in L^{\infty}(\G)$ and $x\in \ell^{\infty}(\h{\G}/\h{\HH})$ extends to an isomorphism between $\ell^{2}(\h{\G}\ltimes \ell^{\infty}(\h{\G}/\h{\HH}))$ and $L^{2}(\G)\otimes \ell^{2}(\h{\G}/\h{\HH})$. Recall that $\RR$ denotes the adjoint of the unitary implementation of $\tau$. Let us endow the latter Hilbert space with the structure of a correspondence from $L^{\infty}(\G)$ to itself induced by the quasi-regular representation, i.e. the left action $\pi_{l}$ and the right action $\pi_{r}$ are given, for every $a\in L^{\infty}(\G)$, by
\begin{equation*}
\pi_{l}(a) = \RR^{*}(a\otimes 1)\RR\text{ and }\pi_{r}(a) = (Ja^{*}J)\otimes 1.
\end{equation*}
Then, the previous isomorphism intertwines these actions with the natural left and right actions of $L^{\infty}(\G)$ on $\ell^{2}(\h{\G}\ltimes \ell^{\infty}(\h{\G}/\h{\HH}))$ (inherited from its identification with $H\otimes \overline{H}$). Thus, the correspondence associated with the quasi-regular representation weakly contains the identity correspondence.
\end{proof}

\subsection{Finite index quantum subgroups}

The simplest source of examples of relatively amenable discrete quantum subgroups is of course \emph{finite index} quantum subgroups.

\begin{de}
A discrete quantum subgroup $\h{\HH}$ of a discrete quantum group $\h{\G}$ is said to have \emph{finite index} if the quotient von Neumann algebra $\ell^{\infty}(\h{\G}/\h{\HH})$ is finite-dimensional.
\end{de}

\begin{prop}
Let $\h{\HH}$ be a finite index quantum subgroup of a discrete quantum group $\h{\G}$. Then, $\h{\G}$ is amenable relative to $\h{\HH}$.
\end{prop}

\begin{proof}
Let us prove that the weight $\theta$ yields an invariant state. Indeed, the operator-valued weight $T : \ell^{\infty}(\h{\G}) \rightarrow \ell^{\infty}(\h{\G}/\h{\HH})$ is finite in that case and we can therefore only consider elements of the form $T(x)$. Using the equality $h_{L} = \theta\circ T$, we have
\begin{equation*}
(\i\otimes\theta)\circ \tau(T(x)) = (\i\otimes (\theta\circ T))\circ \tau(x) = (\i\otimes h_{L})\circ\h{\D}(x) = h_{L}(x).1 = \theta(T(x)).1
\end{equation*}
Thus, $\theta(1)^{-1}\theta$ is an invariant mean for $\tau$.
\end{proof}

\begin{ex}
Let $N\in \N$, let $\G$ be the free orthogonal quantum group $O_{N}^{+}$ and let $u =(u_{ij})_{1\leqslant i, j\leqslant N}$ be its fundamental representation. If $N = 2$, its dual discrete quantum group $\h{\G}$ is amenable, hence it is amenable relative to any quantum subgroup. If $N\geqslant 3$, consider the subalgebra of $L^{\infty}(O_{N}^{+})$ generated by the elements $u_{ij}u_{kl}$ for all $i, j, k, l$. This subalgebra is stable under the coproduct and thus defines a discrete quantum subgroup $\h{\HH}$ of $\h{\G}$ called its \emph{even part}. It is clear that under the usual identification $\Ir(\G) = \N$, $\Ir(\HH)$ corresponds to the even integers and it is not very difficult to see that $\ell^{\infty}(\h{\G}/\h{\HH}) = \C\oplus\C$. Thus $\h{\G}$ is amenable relative to $\h{\HH}$. Note that this quantum group is isomorphic to the quantum automorphism group of $M_{N}(\C)$ (with respect to a suitably chosen trace) according to \cite{banica1999symmetries}.
\end{ex}

As for classical groups, relative amenability becomes equivalent to finite index in the presence of Kazhdan's property (T).

\begin{prop}
Let $\h{\G}$ be a discrete quantum groups with Kazhdan's property (T) as defined in \cite[Def 3.1]{fima2008kazhdan} and let $\h{\HH}$ be a discrete quantum subgroup such that $\h{\G}$ is amenable relative to $\h{\HH}$. Then, $\h{\HH}$ has finite index in $\h{\G}$.
\end{prop}

\begin{proof}
Recall from Proposition \ref{prop:weakcorrespondence} that the correspondence associated with the quasi-regular representation weakly contains the identity correspondence. Since $L^{\infty}(\h{\G})$ has property (T) in the sense of \cite{connes1985property} by \cite[Thm 3.1]{fima2008kazhdan}, it actually contains the identity correspondence. Since any property (T) discrete quantum group is unimodular by \cite[Prop 3.2]{fima2008kazhdan}, we can apply \cite[Lem 7.1]{joita1992property} to conclude that $\mathcal{R}$ has a fixed vector
. This implies by \cite[Lem 2.3]{fima2008kazhdan} that the quotient $\ell^{\infty}(\h{\G}/\h{\HH})$ is finite-dimensional.
\end{proof}

\subsection{Applications}

Theorem \ref{thm:amenableequivalence} links relative amenability for discrete quantum groups and amen\-able equivalence of von Neumann algebras. We can thus now use the work of C. Anantharaman-Delaroche \cite{anantharaman1995amenable} on von Neumann algebras to derive permanence results. However, going back to the quantum group is not always possible at the present state of our knowledge, hence the restriction to unimodular quantum groups in the sequel.

\begin{cor}\label{thm:warelativelyamenable}
Let $\h{\G}$ be a discrete quantum group and let $\h{\HH}$ be a discrete quantum subgroup such that $\h{\G}$ is amenable relative to $\h{\HH}$, then $\Lambda_{cb}(L^{\infty}(\HH)) = \Lambda_{cb}(L^{\infty}(\G))$. If moreover $\h{\G}$ (and consequently $\h{\HH}$) is \emph{unimodular}, then $\Lambda_{cb}(\h{\G}) = \Lambda_{cb}(\h{\HH})$.
\end{cor}

\begin{proof}
It was proved in \cite[Thm 4.9]{anantharaman1995amenable} that amenably equivalent von Neumann algebras have equal Cowling-Haagerup constant. We conclude by \cite[Thm 5.14]{kraus1999approximation}.
\end{proof}

\begin{rem}\label{rem:relativeC*}
A more direct (and C*-algebraic) proof of Corollary \ref{thm:warelativelyamenable} for groups is given in \cite[Prop 12.3.11]{brown2008finite}. However, it is quite ill-suited to the setting of quantum groups since it is based on the use of a section of the quotient which may fail to exist in a reasonable sense in the quantum case, for example if the subgroup is not \emph{divisible} in the sense of \cite[Def 4.1]{vergnioux2011k}.
\end{rem}

A similar statement holds for the Haagerup property.

\begin{cor}\label{rem:relativelyamenablehaagerup}
Let $\h{\G}$ be a \emph{unimodular} discrete quantum group and let $\h{\HH}$ be a discrete quantum subgroup such that $\h{\G}$ is amenable relative to $\h{\HH}$. Then, if $L^{\infty}(\HH)$ has the Haagerup property, $L^{\infty}(\G)$ also has the Haagerup property.
\end{cor}

\begin{proof}
This is a combination of Theorem \ref{thm:amenableequivalence} and \cite[Thm 5.1]{bannon2011some}.
\end{proof}

We end with an application to hyperlinearity. A unimodular discrete quantum group $\h{\G}$ is said to be \emph{hyperlinear} if the von Neumann algebra $L^{\infty}(\G)$ tracially embeds into an ultraproduct of the hyperfinite II$_{1}$ factor.

\begin{cor}
Let $\h{\G}$ be a \emph{unimodular} discrete quantum group and let $\h{\HH}$ be a discrete quantum subgroup such that $\h{\G}$ is amenable relative to $\h{\HH}$. Then, $\h{\G}$ is hyperlinear if and only if $\h{\HH}$ is hyperlinear.
\end{cor}

\begin{proof}
This is a direct consequence of \cite[Thm 3.1]{wu2012co} since, by Remark \ref{rem:popa}, $L^{\infty}(\G)$ is amenable relative to $L^{\infty}(\HH)$.
\end{proof}

\bibliographystyle{amsplain}
\bibliography{../../quantum}

\end{document}